\documentclass[12pt]{amsart}
\usepackage{amsmath,amsthm,amsfonts,graphicx,listings,centernot,fullpage}

\title{Algebraic Properties of Generalized Rijndael-like Ciphers} 
\author{L. Babinkostova $^{1\S}$, Kevin W. Bombardier$^{2}$, Matthew M. Cole$^{3}$, Thomas A. Morrell$^{4}$, and Cory B. Scott$^{5}$}
\thanks{Supported by National Science Foundation grant DMS 1062857}
\thanks{$^{\S}$ Corresponding Author: liljanababinkostova@boisestate.edu}

\usepackage{amsfonts}
\usepackage{amsmath}
\usepackage{amsthm}
\usepackage{amssymb}
\usepackage{graphicx}
\usepackage{multirow}
\usepackage{setspace}
\usepackage{url}
\usepackage{color}

\newcommand{\naturals}{{\mathbb N}}

\newcommand{\field}{{\mathbb F}}

\newcommand{\KK}{\mathcal{K}}

\newcommand{\MM}{\mathcal{M}}
\newcommand{\GG}{\mathcal{G}}

\newcommand{\TT}{\mathcal{T}}
\newcommand{\alt}{\mathcal{A}}
\newcommand{\sym}{\mathcal{S}}

\newcommand{\CC}{\mathcal{C}}
\newcommand{\GF}{\mathrm{GF}}

\newtheorem{theorem}{Theorem}
\newtheorem{lemma}[theorem]{Lemma}
\newtheorem{corollary}[theorem]{Corollary}
\theoremstyle{definition}
\newtheorem{definition}[theorem]{Definition}

\begin{document}
\maketitle

\begin{abstract}  {We provide conditions under which the set of Rijndael-like functions considered as permutations of the state space and based on operations of the finite field $\GF (p^k)$ ($p\geq 2$) is not closed under functional composition. These conditions  justify using a sequential multiple encryption to strengthen generalized Rijndael like ciphers. In \cite {SW}, R. Sparr and R. Wernsdorf  provided conditions under which the group generated by the Rijndael-like round functions based on operations of the finite field $\GF (2^k)$ is equal to the alternating group on the state space. In this paper we provide conditions under which the group generated by the Rijndael-like round functions based on operations of the finite field $\GF (p^k)$ ($p\geq 2$) is equal to the symmetric group or the alternating group on the state space.}
\end{abstract}





\section{Introduction}

An $\mathcal {SP}$-network is an iterated block cipher. This means that a certain sequence of computations, constituting a {\it round}, is repeated a specified number of times. The computations in each round are defined as a composition of specific functions (substitutions and permutations) in a way that achieves Shannon's principle \cite {S} of confusion and diffusion. The {\it Rijndael} block cipher (\cite {DR}, \cite {DRB}) is an example of an $\mathcal {SP}$-network. Rijndael is a block cipher with both a variable block length and a variable key length. The versions for the block size of 128 bits and key length of 128, 192, and 256 bits were adopted by the NIST as the {\it Advanced Encryption Standard} (AES) \cite {NIST}.
Rijndael has a highly algebraic structure. The cipher round transformations are based on operations of the finite field $\GF(2^8)$.  While little research has been done about the structural and algebraic properties of Rijndael before it was adopted as a standard, there has been much research since. Several alternative representations of the AES have been proposed (see, e.g, \cite {BB}, \cite {CP} and \cite {LSWD}) and some group theoretic properties of the AES components have been discovered (see, e.g, \cite {CMRB}, \cite {MR}, \cite {SW} and \cite {We}).

A motivation for investigating the group theoretic structure of a block cipher is to identify and exclude undesirable properties . One such undesirable property is short cycles of the round functions when considered as permutations of the state space. Another undesirable property is non-trivial factor groups of the group generated by the round functions of the cipher. For example, in \cite {P} it was shown that if the group generated by the round functions of a block cipher is imprimitive then this might lead to the design of trapdoors. Some related results about the cycle structure of the AES round functions are given in \cite {LSWD} and \cite{We}.    

Knowing the order of the group generated by the round functions is also an important algebraic question about the security of the cipher, because of its connection to the Markov cipher approach to differential cryptanalysis. In \cite {HSW} it was shown that if the one-round functions of an $s$-round iterated cipher generate the alternating or the symmetric group, then for all corresponding Markov ciphers the chains of  differences are {\it irreducible} and {\it aperiodic}. This means that after sufficiently many rounds of the cipher all differences become equally probable which makes the cipher secure against a differential cryptanalysis attack. In \cite{We}, R. Wernsdorf showed that the round functions of Rijndael over $\GF(2^8)$ generate the alternating group. In \cite {SW}, R. Sparr and R. Wernsdorf  provided conditions under which the group generated by the Rijndael-like round functions which are based on operations on the finite field $\GF (2^k)$ is equal to the alternating group on the state space. Motivated by their work we embark on a formal study of the Rijndael-like functions to  determine the extent to which this and other results in \cite {We}  hold when we consider an arbitrary finite field.  In this paper we provide conditions under which the group generated by the Rijndael-like round functions which are based on operations on the finite field $\GF (p^k)$ ($p\geq 2$) is equal to the symmetric group or the alternating group on the state space. 

Since the adoption of AES as a standard many papers have been published on the cryptanalysis on this cryptosystem. Initially AES survived several cryptanalytic efforts.  The situation started to change in 2009 when \cite {BK} and \cite {BKN} presented a key recovery attack on the full versions of AES-256 and AES-192.  Since then there have been several other theoretical attacks on these versions of AES and AES-128 (see, e.g. \cite{BKR}) as well as  on reduced-round instances of  these versions of AES  (see, e.g. \cite {DKS}).  However, in \cite {BDKS} the authors presented  a key recovery attack on version of AES-256  with up to 10 rounds that is of practical complexity. 

Theoretical attacks against widely used crypto algorithms often get better over time. The crucial question is how far AES is from becoming practically insecure. One way of strengthening AES is through using sequential multiple encryption, as it has been done with DES (see,  \cite{Kaliski},  \cite{CW} and \cite {NIST1}).  If the set of Rijndael round functions is closed under functional composition, then multiple encryption would be equivalent to a single encryption, and so strengthening AES through multiple encryption would not be possible. Thus, it is important to know whether this set is closed under functional composition.  Also, it is important to know  how changing the underlying finite field in AES will impact this property.  In this paper we provide conditions under which the set of Rijndael-like functions considered as permutations of the state space and based on operations of the finite field $\GF (p^k)$ ($p\geq 2$) is not closed under functional composition. 

The idea of examining block ciphers using different binary operations  in their  underlying structure has already been considered. For example, E. Biham and A. Shamir \cite{BS} examined the security of DES against their differential attack when some of the exclusive-or operations in DES are replaced with addition modulo $2^n$. In \cite {PRS} the authors initiated a study of Luby-Rackoff ciphers when the bitwise exclusive-or operation in the underlying Feistel network is replaced by a binary operation in an arbitrary finite group. They showed that in certain cases these ciphers are completely secure against adaptive chosen plaintext and ciphertext attacks and has better time and space complexity if considered over $\GF(p)$ for $p>2$. Although, the study of  the $\mathcal {SP}$-network based ciphers over $\GF(2^r)$ has already been considered (see, e.g. \cite {We}) we are not aware of such study when the underlying operations are the field operations in  $\GF(p^r)$ for $p>2$. 

The paper is organized as follows. In Section 2 we give some background from the theory of permutation groups and finite fields as well as block ciphers. In Section 3 we introduce the generalized Rijndael-like ${\mathcal{SP}}$ network and provide conditions for the parity and the cycle structure of the round functions of such a network when considered as permutations on the state space. Furthermore, we show when the set of round functions in the generalized Rijndael-like $\mathcal{SP}$ network of $s$-rounds do not constitute a group under functional composition. In Section 4 we derive conditions for Rijndael-like round functions such that the group generated by these functions is equal to the alternating group or the symmetric group on the state space. In Section 5 we conclude the paper.

\section {Preliminaries}
\subsection {Iterated block ciphers}
A \emph{cryptosystem} is an ordered 4-tuple  $(\MM,\,\CC,\,\KK,\,T)$  where $\MM$, $\CC$, and $\KK$ are called the \emph {message}(\emph{state}) \emph {space}, the {\it ciphertext space}, and the {\it key space} respectively, and where $T: \MM\times\KK\rightarrow \CC$ is a transformation such that for each $k\in\KK$, the mapping $\epsilon_k:\MM\rightarrow\CC$, called an \emph{encryption transformation}, is invertible.
For any cryptosystem $\Pi=(\MM,\,\CC,\,\KK,\,T)$,  let $\TT_{\Pi}=\{\epsilon_k:k\in\KK\}$ be the set of all encryption transformations. In addition, for any transformation $\epsilon_k\in\TT_{\Pi}$, let ${\epsilon_k}^{-1}$ denote the inverse of $\epsilon_k$.  In a cryptosystem where $\MM=\CC$ the mapping $\epsilon_k$ is a permutation of $\MM$. We consider only cryptosystems for which $\MM = \CC$. The set of all permutations of the set $\MM$ is denoted by $\sym_{\MM}$. Under the operation of functional composition $\sym_{\MM}$ forms a group called \emph{the symmetric group} over $\MM$. 
The symbol $\GG=\langle\TT_{\Pi}\rangle$ denotes the subgroup of $\sym_{\MM}$ that is generated by the set $\TT_{\Pi}$. The group $\GG$ is known as the \emph {group generated by a cipher}. If $\TT_{\Pi} = \GG$, that is the set of permutations $\TT_{\Pi}$ forms a group, then we say the cipher is a group. As $\GG$ is finite by Theorem 3.3 from \cite{G} the cipher is a group if and only if its set of encryption transformations $\TT_{\Pi}$ is a closed under functional composition. For such a cipher, multiple encryption doesn't offer better security than single encryption.
Computing the group $\GG$ generated by a cipher is often difficult. Let $T[k]$ denote the round function of the cipher under the key $k\in \KK$, where $\KK$ denotes the set of all round keys. Let $\tau= \{T[k]\vert k \in \mathcal{K}\}$ be the set of all round functions. The round functions $T[k]$ are also permutations of the message space $\MM$ and it is often easier to compute the group $\GG_{\tau}=\langle\{ T[k]\vert k\in \mathcal{K}\}\rangle$ generated by these permutations. Suppose we have an $s$-round cipher with a key schedule $KS: \KK\rightarrow{\KK^s}$ so that any key $k\in\KK$ produces a set of subkeys $k_i\in \KK$, $1\leq i\leq s$. It is natural then to consider the following three groups relevant to the block cipher:
$$\GG_{\tau}=\langle T[k]\vert k\in\KK\rangle$$
$$\GG_{\tau}^s=\langle T[k_s]T[k_{s-1}]\cdots T[k_1]\vert k_i\in\KK\rangle$$
$$\GG=\langle T[k_s]T[k_{s-1}]\cdots T[k_1]\vert KS(k)=(k_1,k_2,\cdots,k_s)\rangle$$
Thus $\GG_{\tau}$ is the group generated by the round functions and $\GG_{\tau}^s$ is the group generated by the set of all compositions of $s$ (independently chosen) round functions. The group $\GG$ is the group generated by the set of all compositions of $s$ round functions using the key schedule $KS$. This group can also be regarded as the group $\langle\TT_{\Pi}\rangle$ generated by the cipher $\TT_{\Pi}$. It is obvious that $\GG$ is a subgroup of $\GG_{\tau}^s$ which is a subgroup of $\GG_{\tau}$. We will show that $\GG_{\tau}^s$ is in fact a normal subgroup of $\GG_{\tau}$. 

\begin{lemma}\label{NormalSubGrp}
For every $s \in \naturals$, $G_{\tau}^s$ is a normal subgroup of $G_{\tau}$.
\end{lemma}
\begin{proof}
Let $T_k \in G_{\tau}$ and $T_s \circ \cdots \circ T_1 \in G_{\tau}^s$.  We see that
\begin{eqnarray*}
T_k \circ (T_s \circ \cdots \circ T_1) \circ T_k^{-1} & = & T_k \circ (T_s \circ \cdots \circ T_1) \circ (\underbrace{T_k \circ T_k \circ \cdots \circ T_k}_{s-1 \text{ copies}}) \circ\\
&    & \circ (\underbrace{T^{-1}_k \circ T^{-1}_k \circ \cdots \circ T^{-1}_k}_{s-1 \text{ copies}}) \circ T_k^{-1} \\
& = &( T_k \circ T_s \circ \cdots \circ T_2) \circ (T_1 \circ \underbrace{T_k \circ T_k \circ \cdots \circ T_k}_{s-1 \text{ copies}}) \\
&   & \circ (\underbrace{T_k \circ T_k \circ \cdots \circ T_k}_{s \text{ copies}})^{-1} \text{.}
\end{eqnarray*}
It follows that $T_k \circ (T_s \circ \cdots \circ T_1) \circ T_k^{-1} \in G_{\tau}^s$.  This completes the proof.
\end{proof}
Thus the group $G_{\tau}$ generated by the round functions is an upper bound for the group generated by the cipher.

\subsection {Group theoretical background}

In this section we present some background from the theory of permutation groups and finite fields which are used in this paper.
\subsubsection{\bf Permutation groups} 

For a finite set $X$, let $\vert X\vert$ denote the number of elements of $X$. For any nonempty finite set $X$ with $\vert X\vert=n$, the set of all bijective mappings of $X$ to itself is denoted by $\sym_n$ and is called the \emph{symmetric group} on $X$. A permutation $g\in \sym_n$ is a \emph {transposition} if $g$ interchanges two elements $x,y\in X$ and fixes all the other elements of $X\setminus \{x,y\}$. A permutation $g\in \sym_n$ is called an \emph {odd} (\emph{even})  permutation if $g$ can be represented as a composition of an odd (even) number of transpositions\footnote{Note that in this terminology a cycle of even length is an odd permutation, while a cycle of odd length is an even permutation.}.

The set of all even permutations is a group under functional composition and is called the \emph{alternating group} on $X$. The symbol $\alt_n$ denotes the alternating group on a set $X$ with $\vert X\vert=n$. The \emph{degree} of a permutation group $G$ over a finite set $X$ is the number of elements in $X$ that are moved by at least one permutation $g\in G$.
\begin{theorem} \label{simple}  
For $n\geq 5$, the alternating group $\alt_n$ is a simple group.
\end{theorem}

For any subgroup $G\leq\sym_n$, for any $x\in X$, the set $orb_G(x)=\{\phi(x):\phi\in G\}$ is called the \emph{orbit} of $x$ under $G$. The set $stab_G(x)=\{\phi\in G:\phi (x)=x\}$ is called the \emph{stabilizer} of $x$ in $G$. We will make use of the following well-known theorem, often called the Orbit-Stabilizer Theorem.

\begin{theorem}
Let $G$ be a finite group of permutations of a set $X$. Then for any $x\in X$, 
\[
\vert G\vert=\vert orb_G(x)\vert\cdot\vert stab_G(x)\vert
\]
\end{theorem}
Let $l, n$ denote natural numbers such that $0<l\leq n$. A group $G\leq \sym_n$ is called \emph {$l$-transitive} if, for any pair $(a_1,a_2,\ldots,a_l)$ and $(b_1,b_2,\ldots,b_l)$ with $a_i\neq a_j$, $b_i\neq b_j$ for $i\neq j$, there is a permutation $g\in G$ with $g(a_i)=b_i$ for all $i\in \{1,2,\ldots,l\}$. A $1$-transitive permutation group is called \emph {transitive}. 

A subset $B\subseteq X$ is called a \emph {block} of $G$ if for each $g\in G$ either $g(B)=B$ or $g(B)\cap B =\emptyset$. 
A block $B$ is said to be \emph {trivial} if $B\in\{\emptyset,X\}$ or $B=\{x\}$ where $x\in X$. The group $G\leq \sym_n$ is called \emph{imprimitive} if there is a non-trivial block $B\subseteq X$ of $G$; otherwise $G$ is called \emph{primitive}. 
 
We use the following result from \cite{AW} which provides sufficient conditions for a permutation group to be the alternating or the symmetric group.
\begin{lemma} \label{Rodgerslemma}
Suppose $G$ is a primitive permutation group of degree $n$ on a finite set $X$. If $G$ contains a cycle of length $m$ with $2 \leq m \leq (n-m)!$, then $G$ is the alternating or the symmetric group on $X$.
\end{lemma}

\subsubsection{\bf Finite fields}
A structure $(\mathbb {F},+,\cdot)$ is a \emph {field} if and only if both $(\mathbb {F},+)$ is an Abelian group with identity element $0_G$ and $({\mathbb F}\setminus\{0_G\},\cdot)$ is an Abelian groups and the law of distributivity of $\cdot$ over $+$ applies. If the number of elements in $\mathbb F$ is finite, $\mathbb F$ is called a \emph {finite field}; otherwise it is called an \emph {infinite field}.
\begin{definition}
Suppose $\mathbb{F}$ and $\mathbb K $ are fields. If  $\mathbb{F}\subseteq \mathbb K$, then $\mathbb{F}$ is called a {\emph subfield} of $\mathbb K$, or equivalently $\mathbb{K}$ is called an {\emph extension field} of $\mathbb{F}$.
\end{definition}
We can view $\mathbb K$ as a vector space over $\mathbb{F}$ if we define the scalar multiplication as follows
\[
\mathbb{F}\times \mathbb {K} \rightarrow \mathbb {K}
\]
\[
(a,\alpha)\mapsto a\alpha
\]
Suppose the extension field ${\mathbb K}$ of ${\mathbb F}$ is a finite dimensional vector space over ${\mathbb F}$. Let $d=dim_{\mathbb F}(\mathbb K)$ be the dimension of the vector space ${\mathbb K}$ over the field ${\mathbb F}$, and let $\{\alpha_1, \alpha_2,\cdots,\alpha_d\}$ be a basis of the vector space $\mathbb K$ over $\mathbb F$. Then any element $\beta\in \mathbb K$ can be expressed uniquely as a linear combination of $\alpha_1, \alpha_2,\cdots,\alpha_d$ with coefficients in $\mathbb F$

\[
\beta=a_1\alpha_1+a_2\alpha_2+\cdots+a_d\alpha_d
\]
where $a_1,a_2,\cdots, a_d\in \mathbb F$.

In field theory the dimension $d$ of the vector space $\mathbb K$ over $\mathbb F$ is called the \emph{degree} of extension.

It is known that every finite field has order $p^n$ for some prime number $p$ and some positive integer $n$. Such a field is called  a \emph {Galois field} of order $p^n$ and is denoted by $\GF(p^n)$. The following classical fact from the theory of finite fields (see \cite {G}) will be used.
\begin{theorem}\label{fieldorderthm}
$\GF(p^{n_1})\subseteq \GF(p^{n_2})$ if and only if $n_1$ divides $n_2$.
\end{theorem}
It is also known that a finite field ${\mathbb K}$ of order $p^{nd}$ can be constructed as a quotient ring $\frac{{\mathbb F}[x]}{\langle f(x)\rangle}$ where ${\mathbb F}[x]$ is the polynomial ring over the field ${\mathbb F}$ of order $p^n$ and $f(x)\in {\mathbb F}[x]$ is an irreducible polynomial of degree $d$ over ${\mathbb F}$. The field ${\mathbb K}$ is an extension field of degree $d$ of ${\mathbb F}$ i.e., a vector space of dimension $d$ over ${\mathbb F}$. The equivalence classes modulo $f(x)$ in $\frac{{\mathbb F}[x]}{\langle f(x)\rangle}$ of the polynomials $1,x,x^2,\cdots, x^{d-1}$ over ${\mathbb F}$ form a basis of ${\mathbb K}$ viewed as a vector space over the field ${\mathbb F}$.  Thus, using $x^i$ as representative for the equivalence class of $x^i$ modulo $f(x)$ (for $0\le i\le d-1$), the elements in $\mathbb K$ can be represented uniquely as 
\[
a_{d-1}x^{d-1}+ a_{d-2}x^{d-2}+\cdots+a_{2}x^{2}+ax+a_0 
\]
where $a_i\in {\mathbb F}$.
\begin{definition}
A \emph{quadratic field extension} of a field ${\mathbb K}$ is a field extension of degree $2$.
\end{definition}
In the case where a quadratic extension ${\mathbb K}$ arises as the quotient ring $\frac{{\mathbb F}[x]}{\langle f(x)\rangle}$ for an irreducible polynomial $f(x)$ of the form $x^2-c$ with $c$ in ${\mathbb F}$, it is common to replace the equivalence class of $x$ modulo $f(x)$ with the symbol $\sqrt{c}$ when representing the elements of ${\mathbb K}$ as linear combinations of basis elements of the vector space ${\mathbb K}$ over the field ${\mathbb F}$. In this notation, elements of ${\mathbb K}$ are written as $a_0+a_1\sqrt{c}$, where $a_0,a_1 \in \mathbb F$ and ${\mathbb K}$ is usually denoted by $\mathbb F(\sqrt{c})$.

We consider the following function on finite fields.
\begin{definition}
Let $\field$ be a finite field of order $q$ and $\mathbb K$ be an extension field of $\field$ of degree $d$. The \emph{trace} function on $\mathbb K$ with respect to $\field$ is the function $Tr:\mathbb K\rightarrow \mathbb F$ defined by  
\[ 
\textup{Tr}(a) = a + a^q + a^{q^2} + \cdots + a^{q^{d-1}}.
\]
\end{definition}
For any subset $S$ of a field $E$ write $S^{-1}$ for the set $\{s^{-1}\vert 0\neq s\in S\}$. The set $S$ is called  an \emph{inverse-closed} if $S^{-1}\subseteq S$. The inversion map in finite fields is of cryptographic interest, especially when we study the algebraic structure of the ciphers which are based on substitution-permutation networks. 
The following theorem is a result by S. Mattarei  in \cite{Mattarei}.
\begin{theorem} \label{mattareithm}
Let $A$ be a non-trivial inverse-closed additive subgroup of the finite field $E=\GF(p^n)$. Then either $A$ is a subfield of $E$ or else $A$ is the set of elements of trace zero in some quadratic field extension contained in $E$.
\end{theorem}

\begin{lemma}\label{qfelemma}
The number of elements of trace zero in a quadratic field extension $\mathbb K(\sqrt{c})$ of a subfield $\mathbb K \subseteq \GF(p^n)$ is equal to $\vert \mathbb {K} \vert$.
\end{lemma}

\begin{proof}
The set of elements of trace zero in $\mathbb K(\sqrt{c})$ is the set 
\[
\{a_0+a_1\sqrt{c} \; \vert \; a_0,a_1 \in \mathbb K, a_0=0 \}
\]
This set has $\vert {\mathbb K} \vert$ members. 
\end{proof}
\begin{theorem}\label{invclosedaddsubgpthm}
Any non-trivial inverse-closed additive subgroup $H$ of a finite field $\GF(p^n)$ has $p^k$ elements for some $k \vert n$.
\end{theorem}
\begin{proof}
By Theorem \ref{mattareithm}, there are two possibilities: $H$ is a subfield of $\GF(p^n)$, in which case the result follows immediately from Theorem \ref{fieldorderthm}; or $H$ is the set of elements of trace zero in a quadratic field extension $\mathbb K(\sqrt{c})$ of a subfield $\mathbb K \subseteq \GF(p^r)$.  In the latter case, by Theorem \ref{fieldorderthm} we have that $\vert \mathbb K \vert = p^k$ for some $k \vert n$, and Lemma \ref{qfelemma} yields $\vert H \vert = \vert \mathbb K \vert = p^k$.
\end{proof}

\section {Cycle structure of the generalized Rijndael-like round functions}
In this section we show properties of the cycle structure of the round functions of a Rijndael-like $\mathcal SP$-network considered over the field $\GF(p^r)$, which we call \emph {generalized  Rijndael-like functions}. The notation of the generalized Rijndael-like functions  and their component functions will be similar to the notation in \cite {SW}. One exception will be that the underlying field in the generalized Rijndael-like functions and their component functions is the finite field  $\GF(p^r)$ of characteristic $p\geq 2$ instead of $\GF(2^r)$.

Let $m,\,n,\,r $ be positive integers. The symbol $M_{m,n}(\GF(p^r))$ denotes the set of all $m\times n$ - matrices over $\GF(p^r)$. The elements of $\GF(p^r)^{mn}$ are defined as matrices $b\in M_{m,n}(\GF(p^r))$ with the mapping $t:\GF(p^r)^{mn}\rightarrow M_{m,n}(\GF(p^r))$, where $t(a)=b$ is defined by $b_{ij}=a_{ni+j}$, for $0\leq i<m, 0\leq j<n$.   First we start with the analysis of the cycle structure of the component functions in the generalized Rijndael-like function.
\vspace{0.2in}

\subsection {{\bf Analysis of the AddRoundKey-like  function} ($\sigma\left[ k\right]$-function) } 

\begin{definition}
Let $\sigma\left[ k\right]: M_{m,n}(\GF(p^r)) \rightarrow M_{m,n}(\GF(p^r))$ denote the mapping defined by $\sigma\left[ k\right](a)=b$ if and only if $b_{ij}=a_{ij}+k_{ij}$ and $k\in M_{m,n}(\GF(p^r))$ for all $0\leq i <m$, $0\leq j <n$. 
\end{definition}
\begin{lemma}\label{ARKlemma}
Let  $k\in M_{m,n}(\GF(p^r))$ be given. 
\begin{enumerate}
\item  If $p>2$ then $\sigma\left[ k\right]$ is always an even permutation.
\item  If $p=2$ then $\sigma\left[ k\right]$ is an even permutation  if and only if  $r\cdot m\cdot n>1$.
\end{enumerate}
\end{lemma} 
\begin{proof}
If $k = \mathbf{0}$, $\sigma\left[ k\right]$ is the identity permutation.  If $k \ne \mathbf{0}$, then $\sigma\left[ k\right]$ is composed of $p$-cycles.  If $p$ is odd then there are no cycles of even length. If  $p = 2$ then $\sigma\left[ k\right]$ is composed of $2^{rmn-1}$  many $2$-cycles.  
\end{proof}

\subsection{{\bf Analysis of the SubBytes-like function } ($\lambda$-function)}

\begin{definition} 
Let $\lambda: M_{m,n}(\GF(p^r))\rightarrow M_{m,n}(\GF(p^r))$ denotes the mapping defined as a parallel application of $m\cdot n$ bijective S-box-mappings $\lambda_{ij} : \GF(p^r) \rightarrow \GF(p^r)$ and defined by $\lambda(a)=b$ if and only if $b_{ij}=\lambda_{ij}(a_{ij})$ for all $0\leq i<m, 0\leq j<n$.  
\end{definition}
Each S-box mapping consists of an inversion, multiplication by a fixed  $A\in \GF(p^r)$, and addition of a fixed element $B\in\GF(p^r)$ i.e. it is a mapping of the form $Ax^{-1}+B$ where $A,B\in \GF(p^r)$ are fixed. For convenience we define this map on all of $\GF(p^r)$ so that it maps $0$ to $B$, and any nonzero $x$ to $Ax^{-1}+B$.
\begin{lemma}\label{SBlemma}
Let $A\in \GF(p^r)$ be the fixed element used in the S-box mapping $\lambda_{ij}$.  If $p=2$ then the function $\lambda$ is an odd permutation if and only if $r\geq 2$ and $m\cdot n=1$. If  $p>2$ then the function $\lambda$ is an odd permutation if and only if $m$ and $n$ are odd, and either
\begin{enumerate}
\item $p \equiv_{4} 3$, $r$ is odd, and $(p^r-1)/\left\vert\left\langle A \right\rangle\right\vert$ is odd, or 
\item Either $p \equiv_{4} 1$ or $r$ is even, and $(p^r-1)/\left\vert\left\langle A \right\rangle\right\vert$ is even.
\end{enumerate}
\end{lemma}
\begin{proof}
\underline{Analysis of inversion:}  We first consider a single S-box inversion 
\[
   f:\GF(p^r)\rightarrow \GF(p^r): x\mapsto f(x) = \left\{\begin{tabular}{ll}
                                                        $x^{-1}$ & if $x\neq 0$ \\
                                                          $0$    & otherwise.
                                                          \end{tabular}
                                                   \right. 
\]
If we enumerate the elements of $\GF(p^r)$ as $(0,x_1,\,\cdots,\,x_{p^r-1})$, then we can represent $f$ in standard permutation form as 
\[
  f = \left(\begin{tabular}{ccccccc}
                   0 & $x_1$      & $x_2$      & $\cdots$ & $x_i$      & $\cdots$ & $x_{p^r-1}$ \\
                   0 & $x_1^{-1}$ & $x_2^{-1}$ & $\cdots$ & $x_i^{-1}$ & $\cdots$ & $x^{-1}_{p^r-1}$
            \end{tabular}
     \right) 
\]
Writing this in disjoint cycle form we see that $f$ consists entirely of 1-cycles and 2-cycles. The 1-cycles correspond to the $x$ for which $x^2 = 1$ or $x=0$, while 2-cycles correspond to the rest of the $x$'s.

Assume that $p>2$. Since $\GF(p^r)\setminus\{0\}$ is a cyclic group under multiplication, it has only $\phi(2)=1$ element of order $2$, and thus counting the identity also, there are two elements $x$ with $x=x^{-1}$. Thus, there are $p^r-3$ other nonzero elements, and these form 2-cycles in pairs, giving a total of $\frac{1}{2}(p^r-3)$  many 2-cycles in the disjoint cycle decomposition of the $f$ function. If $p=2$ then $p^r-1$ is odd, and so the cyclic group $\GF(2^r)\setminus\{0\}$ (under multiplication) has no elements of order $2$ (since $2$ is not a divisor of $2^r-1$), and so there is only one solution to $x=x^{-1}$ in this case, namely the identity. The remaining $2(2^{r-1}-1)$ non-zero elements contribute $2^{r-1}-1$ disjoint 2-cycles in the cycle decomposition of the $f$ function.

Next we analyze the inversion function as a function over $M_{m,n}(\GF(p^{r}))$.
\begin{itemize}
\item [(a)] {Consider $p>2$.} 
\end{itemize}
When $p>2$, a fixed position (i,j) S-box inversion defined on $M_{m,n}(\GF(p^r))$ still consists of 1-cycles and 2-cycles. The remaining $mn-1$ positions in the $m\times n$ matrices in $M_{m,n}(\GF(p^r))$ can be filled in
 $p^{rmn-r}$ ways, thus producing   
 \[
  \frac{1}{2} (p^{rmn-r})(p^r-1)
  \]
2-cycles over $M_{m,n}(\GF(p^{r}))$, leading to a total of
\begin{equation}
\frac{1}{2}(p^{rmn-r})(p^r-3)
\end{equation}
2-cycles, which is an odd number if $p \equiv_4 1$  or $r$ is even.
\begin{itemize}
\item [(b)] {Consider $p=2$.} 
\end{itemize}
Over $M_{m,n}(\GF(2^{r}))$, a fixed position (i,j) S-box inversion consists of inversion in one position's subfield $\GF(2^r)$ and the identity on all other $(mn-1)$ subfields. Therefore, for every 2-cycle over $\GF(2^r)$, there are $2^{rmn-r}$  many 2-cycles over $\GF(2^{rmn})$.  The total number of 2-cycles is 
\[
\frac{1}{2}(2^{rmn-r})(2^r-2),
\]
which is even if and only if  $mn \ge 2$.
{\flushleft{\underline{Analysis of multiplication by a fixed polynomial in  $\GF(p^r)$:}}}
Multiplication by a fixed polynomial (field element) $A\in\GF(p^r)$ produces cycles of length $\vert\langle A \rangle \vert$ for multiplication with a non-zero field element, and length one for multiplication with the zero element.  Over $M_{m,n}(\GF(p^{r}))$, there are 
\begin{equation}\label{polyeq}
\frac{(p^{rmn-r})(p^r-1)}{\vert\langle A \rangle\vert}
\end{equation}
of these cycles, each of length $\vert\langle A\rangle\vert$ (see equation (\ref{MCeq})).
\begin{itemize}
\item [(a)] {Consider $p>2$.} 
\end{itemize}
Then (\ref{polyeq}) is an odd number if and only if $(p^r-1)/\vert\langle A \rangle \vert$ is odd, in which case the cycle length $\vert\langle A \rangle \vert$ is even. In this case the permutation obtained from multiplication by $A$ is an odd permutation.
\begin{itemize}
\item [(b)] {Consider $p=2$.} 
\end{itemize}
$\vert\langle A\rangle\vert$ is odd, so there are no even-length cycles. In this case the permutation obtained from multiplication by the polynomial $A\in\GF(p^r)$ is an even permutation.
{\flushleft{\underline{Analysis of addition of a constant:}}}
If $p>2$ the addition of a constant is always an even permutation and if $p=2$ then it is even if and only if $m\cdot n\cdot r>1$ (Lemma \ref{ARKlemma}).

From the above, we conclude that for $p$ an odd prime the S-box mapping $\lambda_{ij}$ is odd if $(p^r-1)/|\langle A \rangle |$ is odd, or $p \equiv_4 1$ or $r$ even, but not both. Thus, the function $\lambda$ defined as parallel application of all $m\cdot n$ S-box mappings $\lambda_{ij}$ is odd if and only if each S-box mapping $\lambda_{ij}$ is odd and $m$ and $n$ are odd. For $p=2$ the function  $\lambda$ is odd if and only if $r\geq 2$ and $m\cdot n=1$.
\end{proof}

\subsection{{\bf Analysis of the ShiftRows-like function} ($\pi$-function)}
\begin{definition} 
Let $\pi: M_{m,n}(\GF(p^r))\rightarrow M_{m,n}(\GF(p^r))$ denotes the mapping for which there is a mapping $c: \{0,\ldots, m-1\}\rightarrow \{0,\ldots, n-1\}$ such that $\pi(a)=b$ if and only if $b_{ij}=a_{i(j-c(i))\;mod\;n}$ for all $0\leq i <m$, $0\leq j <n$.  
\end{definition}
We present our analysis of the parity of $\pi$  in two cases according to whether $p$ is an odd prime number or $p=2$.
\begin{lemma}\label{SRlemma} Let $p>2$ be a prime.  The function $\pi$ is an odd permutation if, and only if, $p\equiv_4 3$,  $n$ is even, $r$ is odd, and $gcd(n,c(i))$ is odd for an odd number of $i\in\{0,\,\cdots,\,m-1\}$. 
\end{lemma}
\begin{proof}
The function $\pi$ permutes each row of the state matrix, an element of $M_{m,n}(\GF(p^r))$, by shifting that row by a constant offset. To analyze the parity of the whole permutation, we consider it as the composition of $m$ row permutations. A row permutation shifts a specific row by the corresponding offset, while leaving all other entries of the matrix fixed. Thus for a specific matrix from $M_{m,n}(\GF(p^r))$, such a row permutation leaves $(m-1)n$ entries fixed.

The parity of the function $\pi$  is then computed from the parity of each row permutation by considering the permutation of $M_{1,n}(\GF(p^r))$ corresponding to the restriction of the row permutation that corresponds to the particular row in question. We count the number of even-length cycles (note that an even length cycle is an odd permutation) in the cycle decomposition of this restricted permutation, and then multiply by $p^{r(m-1)n}$ to obtain the number of even length cycles of the row permutation over $M_{m,n}(\GF(p^r))$. 

We first identify the possible lengths of cycles in the cycle decomposition of this permutation, and then we count the number of cycles of each length. From this information and the value of the prime number $p$ we then conclude what is the parity of the permutation $\pi$.

{\flushleft\underline{Analysis of the cycle lengths:}} To determine the possible length of a cycle of the permutation that leaves all entries in the $m\times n$ matrix fixed, except for the $i$-th row, and which shifts the $i$-th row's $n$ entries by $c(i)$ units each, consider all the $n$-vectors whose entries are elements of $\GF(p^r)$. A typical such vector is of the form $(x_0,\cdots,x_{n-1})$ where the $x_j$ are elements of $\GF(p^r)$. A single application of this permutation maps as follows:
\[
  (x_0,\cdots,x_{n-1})\mapsto (x_{n-c(i)+0 \,mod\, n}, \cdots, x_{n-c(i) + n-1 \,mod\, n}).
\]
And $k$ iterations of this permutation maps as follows:
\[
  (x_0,\cdots,x_{n-1})\mapsto (x_{k\cdot(n- c(i))+0 \,mod\, n}, \cdots, x_{k\cdot(n-c(i)) + n-1 \,mod\, n}).
\]
The least $k>0$ which, for any $n$-vector $(x_0,\cdots,x_{n-1})$ of elements of $\GF(p^r)$ produces 
\[
  (x_{k\cdot(n-c(i))+0 \,mod\, n},\,\cdots,\, x_{k\cdot(n-c(i)) + n-1 \,mod\, n}) = (x_0,\cdots,\,x_{n-1})
\]
gives the order of the cyclic group $G$ generated by this row permutation. For this $k$ we have 
\[
  k\cdot(n-c(i)) \equiv 0 \,mod\, n
\]
meaning $k\cdot c(i)$ is a common multiple of $c(i)$ and $n$. By minimality of $k$, this is the least common multiple of $c(i)$ and $n$, which is $\frac{n\cdot c(i))}{gcd(n,c(i))}$ and thus $k = \frac{n}{gcd(n,c(i))}$.

By the Orbit-Stabilizer Theorem we see that for any $n$-vector $(x_0,\cdots,\, x_{n-1})$ we have
\[
  \frac{n}{gcd(n,c(i))} = \vert G\vert = \vert orb_G((x_0,\cdots,x_{n-1}))\vert\cdot\vert stab_G((x_0,\cdots,x_{n-1}))\vert.
\]
But the orbit of $(x_0,\cdots,x_{n-1})$ ``is" the cycle containing $(x_0,\cdots,x_{n-1})$ in the disjoint cycle decomposition of this row permutation. And the length of this cycle is thus a factor of $\frac{n}{gcd(n,c(i))}$. 

For the factor $d=1$, a fixed point is built by taking a vector $(x_0,\cdots,x_{gcd(n,c(i))-1})$, and concatenating it $\frac{n}{gcd(n,c(i))}$ times to form a vector of length $n$. There are $p^r$ choices of each of the $x_i$, and thus $p^{r\cdot gcd(n,c(i))}$  many $n$-vectors with orbit length equal to $1$. 
{\flushleft {\bf Claim 1:}} For each factor $d>1$, there is an $n$-vector $(x_0,\cdots,x_{n-1})$ for which the orbit length is $d$. Fix $f$ such that $d\cdot f = \frac{n}{gcd(n,c(i))}$ and choose distinct elements $x, y\in\GF(p^r)$. Consider the $n$-vector which consists of the concatenation of $f$ copies of the vector $(y,\, \cdots ,\, y,\,x)$ which has only one entry equal to $x$,
\[
  (y,\,\cdots,\,y,\,x)\frown (y,\,\cdots,\,y,\,x) \frown \cdots \frown (y,\,\cdots,\,y,\,x).
\]
Note that the vector $(y,\, \cdots ,\, y,\,x)$ has length $d\cdot gcd(n,c(i))$. 

Consider the last $x$ of this $n$-vector. After a minimum number of $t$ applications of the permutation, it is in a position of an $x$ in the $n$-vector. Then  
\[
t=lcm(d\cdot gcd(n,c(i)), c(i))=\frac{d\cdot gcd(n,c(i))\cdot c(i)}{gcd(d\cdot gcd(n,c(i)),c(i))}
\]
As $d$ divides $\frac{n}{gcd(n,c(i))}$ if follows that $gcd(d,c(i))$ divides $gcd(\frac{n}{gcd(n,c(i))},c(i)) $.  Since $gcd(\frac{n}{gcd(n,c(i))},c(i)) =1$  we have that 
\[
gcd(d\cdot gcd(n,c(i)),c(i)) = gcd(n,c(i))
\]
It follows that $t = d\cdot c(i)$ applications of the permutation has this $n$-vector as fixed point. Any iteration of this $d\cdot c(i)$-iterate has this $n$-vector as fixed point, and the order of this $d\cdot c(i)$-iterate is 
\begin{eqnarray*}
 \frac{\frac{n}{gcd(n,c(i))}}{gcd(d\cdot c(i),\frac{n}{gcd(n,c(i))})} & = & \frac{\frac{n}{gcd(n,c(i))}}{gcd(d,\frac{n}{gcd(n,c(i))})} \\
    & = & \frac{\frac{n}{gcd(n,c(i))}}{d} \\
    & = & \frac{n}{d\cdot gcd(n,c(i))}\\
    & = & f.
\end{eqnarray*}
It follows that 
\[
  \vert stab_G((y,\,\cdots,\,y,\,x)\frown (y,\,\cdots,\,y,\,x) \frown \cdots \frown (y,\,\cdots,\,y,\,x))\vert = f,
\]
and thus the orbit has $d$ elements, meaning that in the cycle decomposition of the permutation the cycle containing this vector has length $d$. This completes the proof of Claim 1, and establishes all occurring cycle lengths for this permutation.

{\flushleft\underline{Analysis of the number of cycles of a given length:}} Fix a divisor $d'$ of $\frac{n}{gcd(n,c(i))}$. We now count the number of cycles of length exactly $d'$ in the cycle decomposition of the given permutation.
As observed before, for $d'=1$ there are exactly $p^{r\cdot gcd(n,c(i))}$ cycles of length $1$ for this permutation. Now consider the case when $d'>1$. It can be shown that if an $n$-vector $(x_0,\cdots,x_{n-1})$ has an orbit of length dividing $d'$, then it is a concatenation of a number of copies of a vector $(y_1,\cdots,y_{d'\cdot gcd(n,c(i))})$. The total number of such vectors  that can be constructed using the elements of $\GF(p^r)$ is $p^{r\cdot d'\cdot gcd(n,c(i))}$.  But  for $d'>1$ many of these ($d'\cdot gcd(n,c(i))$) vectors have orbits whose cardinality is a proper divisor of $d'$ and thus should be excluded from the count of items producing cycles of length exactly $d'$. Notice that for $d$ a divisor of $d'$ the vectors producing orbits of cardinality $d$ are obtained by concatenating the vector $(y_1,\cdots,y_{d\cdot gcd(n,c(i))})$ the appropriate number of times. The vectors among ones of the form $(z_1,\cdots,z_{d'\cdot gcd(n,c(i))})$ to be excluded are those obtained by concatenating $\frac{d'}{d}$ copies of a vector $(y_1,\cdots,y_{d\cdot gcd(n,c(i))})$ to obtain the vector $(z_1,\cdots,z_{d'\cdot gcd(n,c(i))})$. Let  $N(d')$ denote the number of ($d'\cdot gcd(n,c(i))$) vectors that produce cycles of length exactly $d'$. Thus, $N(1) = p^{r\cdot gcd(n,c(i))}$.
For $d'>1$  we find that
\[
  N(d') = p^{r\cdot d'\cdot gcd(n,c(i))} - \sum_{d\vert d',\, d\neq d'} N(d)
\]
Alternately this can be written 
\[
  p^{r\cdot d'\cdot gcd(n,c(i))} = \sum_{d\vert d'}N(d).
\]
By the M\"obius inversion formula (Theorem 2 on p. 20 of \cite{IR}) we have
\begin{equation}\label{mobiuseq}
  N(d') = \sum_{d\vert d'}p^{r\cdot d\cdot gcd(n,c(i))}\mu\Bigg{(}\frac{d'}{d}\Bigg{)}.
\end{equation}
Note that since each orbit contains exactly $d'$ elements, the number of disjoint cycles in the cycle decomposition of the permutation contributed by these vectors is $\frac{N(d')}{d'}$. The question is whether the number $\frac{N(d')}{d'}$ is even, or odd. Since a cycle of odd length is an even permutation, the answer to this question is relevant only when $d'$ is even.
Let $d'$ be even and have prime factorization
\begin{equation}\label{dfactors}
  d' = 2^a \cdot p_1^{s_1}\cdot\dots\cdot p_t^{s_t}, \hspace{0.1in} a>0.
\end{equation}
Since we are interested in only the parity of $\frac{N(d')}{d'}$, we seek to determine if 
\begin{equation}\label{parityonly}
 N(d') \,mod\, 2^{a+1}
\end{equation} 
is zero, or positive.

Consider $\mu(\frac{d'}{d})$ for an even $d'$ and a factor $d$ of $d'$. By the definition of $\mu$, the only case when $\mu(\frac{d'}{d})$ is non-zero is when $\frac{d'}{d}$ is $1$, or else square free (\emph{i.e.}, a product of distinct prime numbers). In each of these cases the power of $2$ that divides into $d$ is at least $2^{a-1}$, so that the factor $p^{r\cdot d\cdot gcd(n,c(i))}$ of the term corresponding to the factor $d$ is of the form $v^{2^{a-1}}$ where $v$ is an odd number if $p$ is an odd prime number, and even otherwise. 

Let $a >1$.  Then for any odd number $v$ we have that $v^{2^{a-1}} \equiv 1 \, mod \,2^{a+1} $, by Theorem 2$^{\prime}$ in Chapter 4.1 of \cite{IR}. Then the equation (\ref{parityonly}) reduces to 
\begin{eqnarray*}
 \sum_{d\vert d'}p^{r\cdot d\cdot gcd(n,c(i))}\mu\Bigg{(}\frac{d'}{d}\Bigg{)} \, mod\, 2^{a+1}
   & = & \sum_{d\vert d'} 1\cdot \mu\Bigg{(}\frac{d'}{d}\Bigg{)} \, mod\, 2^{a+1} =0
\end{eqnarray*}
since for any integer $d'>1$ we have, by Proposition 2.2.3 on p. 19 of \cite{IR}, that  $\sum_{d\vert d'}\mu(d) = 0$.

Next, consider $a=1$. We need to analyze the following two cases.\\
{\bf Case 1:} $r$ is even or $p\equiv_4 1$.
Since for each odd number $v$ we have $v^2\equiv_{4} 1$ and since we have $a=1$ in equation (\ref{dfactors}), the equation (\ref{parityonly}) reduces to
\begin{eqnarray*}
\sum_{d\vert d'}p^{r\cdot d\cdot gcd(n,c(i))}\mu\Bigg{(}\frac{d'}{d}\Bigg{)} \, mod\, 4
   & = & \sum_{d\vert d'} 1\cdot \mu\Bigg{(}\frac{d'}{d}\Bigg{)} \, mod\, 4=0
\end{eqnarray*}
using Proposition 2.2.3 on p. 19 of \cite{IR} as in the previous case and the fact that $a=1$ and $p \equiv_4 1$ or $r$ is even.
This concludes the argument that if $p$ is a prime number such that $p \equiv_ 4 1$, or if $r$ is even, then for each $i$ the permutation shifting each item in the $i$-th row of $M_{m,n}(\GF(p^r))$ by $c(i)$ units is an even permutation. As a result the function  $\pi$  is a composition of even permutations, and thus is an even permutation in this case.\\
{\bf Case 2:} $r$ is odd and $p\equiv_{4} 3$. We will start analyzing this case by first assuming that  $d'>2$. The factors of $d'$ are either of the form $2d$ where $d$ is odd, or $d$ where $d$ is odd.  
{\flushleft\underline {Part 1:}}  Factors of the form $2d$, where $d$ is odd.  Since for each odd number $v$,  $v^2\equiv_{4} 1$ we have the following
\begin{eqnarray*}
  \sum_{2d\vert d'}p^{r\cdot 2d\cdot gcd(n,c(i))}\mu\Bigg{(}\frac{d'}{2d}\Bigg{)} \, mod\, 4
   & =& \sum_{d\vert \frac{d'}{2}} p^{r\cdot 2d\cdot gcd(n,c(i))}\cdot \mu\Bigg{(}\frac{d'}{2d}\Bigg{)} \, mod\, 4 \\
   &= & \sum_{d\vert \frac{d'}{2}} 1 \cdot \mu\Bigg{(} \frac{d'}{2d} \Bigg{)} \, mod\, 4 =0
\end{eqnarray*} 
again using Proposition 2.2.3 on p. 19 of \cite{IR} as before.
{\flushleft\underline {Part 2:}}   Factors of the form $d$, where $d$ is odd. 
First note that if $v$ is an odd number such that $v \equiv_{4} 3$, then for any odd number $r$,  and $v^r \equiv_{4} 3$. Using this observation and the fact that $a=1$ in equation (\ref{dfactors}), the equation (\ref{parityonly}) reduces to
\begin{eqnarray*}
 \sum_{d\vert \frac{d'}{2}}p^{r\cdot d\cdot gcd(n,c(i))}\mu\Bigg{(}\frac{d'}{d}\Bigg{)} \, mod\, 4
   & = &  
  \sum_{d\vert \frac{d'}{2}} 3^{gcd(n,c(i))} \cdot\Bigg{(}- \mu\Bigg{(}\frac{d'}{2d}\Bigg{)}\Bigg{)} \, mod\, 4 \\
  &= &  
  (-3^{gcd(n,c(i))}) \cdot \sum_{d\vert \frac{d'}{2}} \mu\Bigg{(}\frac{d'}{2d}\Bigg{)} \, mod\, 4=0
\end{eqnarray*} 
Here we used the fact that $\mu$ is multiplicative, so that for odd $w$, $\mu(2w) = \mu(2)\mu(w) = -\mu(w)$, and we again used Proposition 2.2.3 on p. 19 of \cite{IR}.
Taking Part 1 and Part 2 together, we obtain for $d'>2$ that   $N(d') \equiv_4 0$.

Next, assume that $d'=2$. Then  $\frac{N(d')}{d'}$ reduces to
\begin{eqnarray*}
\frac{p^{r\cdot 1\cdot gcd(n,c(i))}\mu(2) + p^{r\cdot 2\cdot gcd(n,c(i))}\mu(1)}{2} 
   & = &  
\frac{p^{r\cdot gcd(n,c(i))}\cdot(p^{r\cdot gcd(n,c(i))} -1)}{2}
\end{eqnarray*}
Since $p$ is odd, the parity of this quantity depends entirely on the parity of  $\frac{p^{r\cdot gcd(n,c(i))} -1}{2}$, which in turn depends on the parity of $r\cdot gcd(n,c(i))$. For this we consider the parity of $\frac{(4k+3)^m-1}{2}$ (since $p\equiv_ 4  3$). By the Binomial Theorem $(4k+3)^m$ has the form $3^m + 4x$ for an appropriate integer $x$, and so $\frac{(4k+3)^m-1}{2} = \frac{3^m + 4x - 1}{2}$, and the parity of this quantity depends on the parity of $\frac{3^m-1}{2}$. Applying the Binomial Theorem to $3^m = (2+1)^m$, we see that $3^m$ is of the form $1 + 2m + 4x$ for an appropriate integer $x$. Thus, $\frac{3^m-1}{2}$ is of the form $\frac{2m+4x}{2}$, which is even if, and only if, $m$ is even. Thus, as $r$ is odd, we find that  $\frac{N(2)}{2}\equiv _2 0$ if $gcd(n,c(i))$ is even and $\frac{N(2)}{2}\equiv _2 1$ if $gcd(n,c(i))$ is odd.  Since for divisors $d'>2$ of $\frac{n}{gcd(n,c(i)}$ we have $\frac{N(d')}{d'}$ even, if follows that when $p \equiv _4 3$ the row permutation is even if, and only if, $r\cdot gcd(n,c(i))$ is even. Since the function $\pi$  is a composition of these row permutations we see that for $p\equiv _4 3$, we have that $\pi$ is an odd permutation if and only if $r\cdot{gcd(n,c(i))}$ is odd for an odd number of $i$ and even for the remaining values of $i$.
\end{proof}

\begin{lemma}\label{2isp} 
The function $\pi: M_{m,n}(\GF(2^r))\rightarrow M_{m,n}(\GF(2^r))$ is an odd permutation if, and only if, $m\cdot r\cdot gcd(n,c(0))=1$ and $n=2$.
\end{lemma}
\begin{proof} We analyze separately the case when $m>1$ and $m=1$.

{{\bf Case 1:} Let $m>1$.} The number $2^{r(m-1)n}$ is an even number, and each cycle length of the permutation  $\pi$ appears a multiple of $2^{r(m-1)n}$ times in its cycle decomposition. Thus in this case $\pi$ is an even permutation.

{{\bf Case 2:} Let $m=1$.} Then the function $\pi$ is a single row permutation, and the factor $2^{r(m-1)n}$ is equal to $1$, so that the parity argument when $m>1$ does not apply.  Once again apply the equations (\ref{mobiuseq}) and (\ref{parityonly}) for $p=2$. Considering a factor $d'$ of $\frac{n}{gcd(n,c(0))}$ with factorization as in equation (\ref{dfactors}), we distinguish again between the cases $a>1$ and $a=1$.

For $a>1$ we have  $n>2$ and the factors $2^{r\cdot d\cdot gcd(n,c(i))}$ in the nonzero terms of (\ref{mobiuseq}) have $2^{a-1}$ as a divisor of $d$. Write $d = k_d\cdot 2^{a-1}$. We have
\[
   2^{r\cdot d\cdot gcd(n,c(i))} =  2^{r\cdot k_d\cdot 2^{a-1}\cdot gcd(n,c(i))}
\]
which for each nonnegative integer $a$ is divisible by $2^{2^{a-1}}$, which in turn is divisible by $2^{a+1}$. Thus we find from equation (\ref{parityonly}) that $N(d')\equiv \mu(d')2^{r\cdot gcd(n,c(i))} \, mod\, 2^{a+1}$. 
But since $a>1$ we must have $\mu(d')=0$. It follows that $\frac{N(d')}{d'}$ is even in this case.

For $a=1$ we see  that the only contributing terms to the parity of the $i$-th row permutation are of the form 
\[
  N(d') = \mu(d')2^{r\cdot gcd(n,c(0))} \,mod\, 4
\]
where $d'$ is an even squarefree factor of $\frac{n}{gcd(n,c(0))}$. If $r\cdot gcd(n,c(0))>1$ then $N(d') \equiv_ 4 0$ and the  factor $d'$ of $\frac{n}{gcd(n,c(0))}$ contributes an even number of cycles of even length to the cycle decomposition of the row permutation. We see that for $m\cdot r\cdot gcd(n,c(0)) > 1$ the function $\pi$ is an even permutation.

Finally consider the case when $m\cdot r\cdot gcd(n,c(0))=1$. For $d'>2$ a squarefree even factor of $\frac{n}{gcd(n,c(0))}$, we have that $N(d') =\mu(d')2 \equiv_4 2$. Suppose that $n$ has $x+1$ distinct prime factors, including $2$. Thus, as $d'>2$, we have $x>0$. The number of squarefree even factors of $\frac{n}{gcd(n,c(0))}$  larger than $2$  is $2^x-1$, an odd number. Thus the squarefree even factors of $\frac{n}{gcd(n,c(0))}$ larger than $2$ contribute an odd number of even length cycles to the cycle decomposition of the permutation $\pi$. To complete the count of the number of cycles of even length in the cycle decomposition of ShiftRows, we must still consider $\frac{N(2)}{2}$. By equation (\ref{mobiuseq}),
\begin{eqnarray*}
  \frac{N(2)}{2} =  \frac{2\mu(1) + 2^{2}\mu(1)}{2}=\frac{2(2-1)}{2}=1.             
\end{eqnarray*}
In conclusion we find that for even $n>2$ the function $\pi$ is an even permutation. For $n=2$, $m=1$, $r=1$ and $c$ odd, the $\pi$ permutation has one $2$-cycle,  and two fixed points, and is thus an odd permutation.
\end{proof}

\begin{lemma}\label{SRlinear}
The function $\pi$ is a linear transformation of the vector space $M_{m,n}(\GF(p^r))$ over the field $\GF(p^r)$. 
\end{lemma}
\begin{proof}
We use the ideas in the proof of Lemma \ref{SRlemma}. The space $M_{m,n}(GF(p^r))$ can be viewed as a direct sum 
\[
V=V_1 \oplus \cdots \oplus V_{m}.
\]
where each $V_i$ is $M_{1,n}(\GF(p^r))$, the space of $n\times 1$ row vectors over the field $\GF(p^r)$. 

Now consider $V_i$ as the space of $i$-th rows of members of $M_{m,n}(\GF(p^r)$. With $c(1),\cdots,c(n)$ defined as before, a single application of the $\pi$ permutation maps the $i$-th row
\[
  (x_0,\cdots,x_{n-1})\mapsto (x_{n-c(i)+0 \,mod\, n}, \cdots, x_{n-c(i) + n-1 \,mod\, n})
\]
Define from the $n\times n$ identity matrix the matrix $C_i$ by letting this mapping act on each of the columns of the identity matrix as if it were the $i$-th row. As the reader could verify, this matrix $C_i$ has the property that
\[
  \lbrack x_0,\cdots,x_{n-1}\rbrack \cdot C_i = \lbrack x_{n-c(i)+0 \,mod\, n}, \cdots, x_{n-c(i) + n-1 \,mod\, n} \rbrack.
\]
Note that $C_i$ is a linear transformation of the vector space $M_{1,n}(\GF(p^r))$ over the field $\GF(p^r)$.

Now the  function $\pi$ on $M_{m,n}(\GF(p^r))$ can be viewed as the direct sum of $C_1\oplus\cdots\oplus C_m$, where for $v_1+v_2+\cdots+v_m \in V$ we have
\[
  C_1\oplus\cdots \oplus C_m(v_1+\cdots+v_m) = C_1\cdot v_1 + \cdots + C_m\cdot v_m,
\]
which is a linear transformation on $M_{m,n}(\GF(p^r))$.
\end{proof}

\subsection{{\bf Analysis of the MixColumns-like function}  ($\rho$-function) }

\begin{definition} 
Let $\rho: M_{m,n}(\GF(p^r))\rightarrow M_{m,n}(\GF(p^r))$ is mapping defined as the parallel application of $n$ ``column" mappings $\rho_j:M_{m,1}(\GF(p^r)) \to M_{m,1}(\GF(p^r))$ defined by $\rho(a)=b$ if and only if $b_j=\rho_j(a_j)$ for all $0 \le j < n$, where each $\rho_j$ is given by $\rho_j(x)=C\cdot x$ for all $x \in M_{m,1}(\GF(p^r))$, where $C\in M_{m,m}(\GF(p^r)) $  is an invertible diffusion matrix. 
\end{definition}
\begin{lemma}\label{MClinear} The function $\rho$  is a linear transformation of $M_{m,n}(\GF(p^r))$.
\end{lemma}
\begin{lemma}\label{MClemma}
Let $C\in M_{m,m}(\GF(p^r))$ be an invertible diffusion matrix and $n>1$.  Then the  function $\rho$ is an odd permutation if and only if  $p$, $n$, and $\frac{p^{rm}- 1}{\vert\langle C \rangle\vert}$ are odd.
\end{lemma}
\begin{proof}
Consider the function $\rho$ as a composition of $n$ permutations $\rho_j$, each of which multiplies the $j$\textsuperscript{th} column by the invertible $m\times m$ matrix $C$ over $\GF(p^r)$ and fixes the other $n-1$ columns.  Fix $j\in\naturals$. Then $\rho_j$ produces cycles of length $\vert\langle C \rangle\vert$.  Of the $p^{rm}$ possible states of the $j$\textsuperscript{th} column all but the fixed points of $C$, which is only the all-$0$ column, are members of cycles.  Note that for any state of the $j$\textsuperscript{th} column, there correspond $p^{rm(n-1)}$ states of the entire matrix.  Therefore, over $M_{m,n}(\GF(p^{r}))$, the permutation $\rho_j$ consists of 
\begin{equation}\label{MCeq}
\frac{p^{rm(n-1)} (p^{rm}-1)} {\vert\langle C \rangle\vert}
\end{equation}
cycles of length $\vert\langle C \rangle\vert$.  This number of cycles is odd if and only if $p$ is odd and $\frac{p^{rm}-1}{\vert\langle C \rangle\vert}$ is odd (in this case $\vert\langle C \rangle\vert$ is even). Note that for only an odd number of $\rho_j$'s would their composition then be an odd permutation, meaning $n$ must be odd. 
\end{proof}
Note that for $p=2$ the $\rho$ function is odd if and only if $n=1$. Additionally, for $n=1$ and $p>2$ the $\rho$ function is odd if and only if  $\frac{p^{rm}-1}{\vert\langle C \rangle\vert}$ is odd.

\subsection{{\bf Analysis of the generalized Rijndael-like round functions} }

\begin{definition}
Let $m,n,r>0$ be natural numbers and $k\in \KK$. The mapping $T{[k]}: M_{m,n}(\GF(p^r))\rightarrow M_{m,n}(\GF(p^r))$ defined as $T[k]=\sigma[k] \circ \rho \circ \pi \circ \lambda$ is called a generalized Rijndael-like round function.
\end{definition}
\begin{corollary}
Let $p$ be an odd prime. For each $k\in\KK$, the generalized Rijndael-like round function $T{[k]}$ is an odd permutation if and only if exactly one of  the functions  $\lambda$, $\rho$, and $\pi$ is  odd.
\end{corollary}
\begin{proof}
By Lemma \ref{ARKlemma} each $\sigma[k]$ is an even permutation. By the definition the function $T[k]$ is odd if and only if each of $\lambda$, $\rho$ and $\pi$ is odd, or else exactly one of these three functions is odd. By Lemmas  \ref{SBlemma}, \ref{SRlemma} and \ref{MClemma} these three functions cannot simultaneously be of the same parity.
\end{proof}
\begin{corollary}\label{n>2}
For $n>2$  the Rijndael-like round function $T{[k]}: M_{m,n}(\GF(2^r))\rightarrow M_{m,n}(\GF(2^r))$ is an even permutation. 
\end{corollary}
\begin{corollary} \label{n=2}
The Rijndael-like round function $T{[k]}: M_{m,2}(\GF(2^r))\rightarrow M_{m,2}(\GF(2^r))$ is an even permutation if and only is $\pi$ is even. 
\end{corollary}
\begin{corollary} \label{n=1}
The Rijndael-like round function $T{[k]}: M_{m,1}(\GF(2^r))\rightarrow M_{m,1}(\GF(2^r))$ is an even permutation if and only if $\sigma[k]$ is odd or $\lambda$ is odd. 
\end{corollary}
Note that when $n=1$ and $m=2$ the Rijndael-like round function $T{[k]}$ is an odd permutation.
\begin{definition}
Let $m,\, n,\, r>0$ be natural numbers and $k \in \KK$.  For $s>1$ and  $2 \leq i \leq s$  the  mapping $T_s[k]: M_{m,n}(\GF(p^r))\rightarrow M_{m,n}(\GF(p^r))$ defined as  
\[
T_s[k] = \sigma[k_{s+1}] \circ \pi \circ \lambda \circ (\sigma[k_s] \circ \rho \circ \pi \circ \lambda)\circ\cdots \circ  (\sigma[k_2] \circ \rho \circ \pi \circ \lambda)\circ \sigma[k_1]
\]
where $\{k_i: 1\leq i\leq s\}$ is the set of subkeys produced by the key $k$ is called $s$-round generalized Rijndael-like function.
\end{definition}
The AES  as well as the actual Rijndael \cite {DRB} are  special $s$-round Rijndael-like functions for $m = n = 4$,  $r =8$ , $p=2$ and $s=10,\,12,\, \mbox{or } 14$ (depending on key size). \\
\begin{theorem}\emph {\cite{SW}} \label{SWtheorem}
Let $mn>2$ and $r\geq 2$ be natural numbers. Then the  $s$-round Rijndael-like function
\[
T_s[k]: M_{m,n}(\GF(2^r))\rightarrow M_{m,n}(\GF(2^r))
\]
 is an even permutation.
\end{theorem} 
Using corollaries \ref{n>2}, \ref{n=2} and \ref{n=1} we have the following generalization  of the theorem above.
\begin{theorem}
For $n>2$ the $s$-round Rijndael-like function
\[
T_s[k]: M_{m,n}(\GF(2^r))\rightarrow M_{m,n}(\GF(2^r))
\]
 is an even permutation.
\end{theorem} 
\begin{corollary}
The  $s$-round Rijndael-like function $T_s[k]: M_{m,2}(\GF(2^r))\rightarrow M_{m,2}(\GF(2^r))$ is an even permutation if and only if $\pi$ is odd and $s$ is even or $\pi$ is even.
\end{corollary} 
\begin{corollary}
The  $s$-round Rijndael-like function $T_s[k]: M_{m,1}(\GF(2^r))\rightarrow M_{m,1}(\GF(2^r))$ is an even permutation if and only if $\sigma$ is odd or $\lambda$ is odd or else $s$ is even.
\end{corollary} 
The proofs of the theorems below are omitted as they follow directly from the above theorems about the parity of the functions $\sigma,\,\rho,\, \lambda$ and $ \pi$. 
\begin{theorem}\label{fullpermutations}
Let $p>2$ be a prime. Then the $s$-round generalized Rijndael-like function  $T_s[k]$ is an odd permutation if
\begin{itemize}
\item [(i)] $s$ is even, and $\rho$ is  odd, or else
\item [(ii)] $s$ is  odd, and either $\pi$ or $\lambda$ is odd.
\end{itemize}
 \end{theorem}
\begin{corollary}
Let $p>2$ be a prime.  Then the set of $s$-round Rijndael-like functions do not form a group if
\begin{itemize}
\item  [(i)] $s$ is even, and $\rho$ is  odd, or else
\item  [(ii)] $s$ is  odd, and either $\pi$ or $\lambda$ is odd.    
\end{itemize}
\end{corollary}

\section{Groups generated by the generalized Rijndael-like round functions}

In this section we show properties of groups generated by the round functions of the Rijndael-like $\mathcal {SP}$-network. We provide conditions under which the group generated by the generalized  Rijndael-like round functions based on operations of the finite field $\GF (p^k)$ ($p\geq 2$) is equal to the symmetric group or the alternating group on the state space. Some of the techniques that we use for this result appear in \cite{C}.

In our analysis of this group note that by Lemmas \ref{SRlinear} and \ref{MClinear}, the functions $\rho$ and $\pi$ appearing in $T[k]=\sigma[k] \circ \rho \circ \pi \circ \lambda$ are both linear. Thus the map $\alpha=\rho \circ \pi$ is a linear transformation. 

The space $V=M_{m,n}(\GF(p^r))$ is a direct sum 
\[
  V=V_1 \oplus \cdots \oplus V_{mn}.
\]
where each $V_i$ has dimension $r$ over $\GF(p)$. For any $v \in V$ we write
\[
v=v_1+\dots+v_{mn}
\]
where $v_i \in V_i$. Also, we consider the projections $\textrm{Proj}_i : V \to V_i$ onto $V_i$ given by $\textrm{Proj}_i(v)=v_i$. 
\begin{definition}
We say that $\gamma :V \to V$ is a piecewise Galois field inversion if for all $v \in V$, $\gamma(v):=(v_1)^{\epsilon_1} \oplus \dots \oplus (v_{mn})^{\epsilon_{mn}}$, where $\epsilon_{mn}\in\{-1,1\}$ is such that
\[
  \epsilon_i = \left\{\begin{tabular}{ll}
                     -1 & if $v_i\neq 0$\\
                      1 & otherwise
                     \end{tabular}
              \right. 
\] 
\end{definition}

\begin{lemma}\label{1rdlemma}
Let $\gamma_i$ denotes the restriction of $\gamma$ to $V_i$ and let $r>4$. Then  
\begin{itemize}
\item [(1)] { $\gamma(0) = 0$ and $\gamma^2$ is the identity map.}
\item [(2)] {For all $i \in \mathbb{Z}_{mn}$ and}
\begin{itemize}
\item[(a)]{For all $v \in V_i$ where $v \neq 0$, the image of the map $V_i \to V_i$ which maps $x\mapsto \gamma_i(x + v) - \gamma_i(x)$ has size greater than $p^{r-2}$, and}
\item[(b)]{If a subspace of $V_i$ is invariant under $\gamma_i$ then it has codimension at least $3$.}
\end{itemize}
\end{itemize}
\end{lemma}

\begin{proof}
The  condition (1) is satisfied by construction of $\gamma$. 

{\flushleft{\underline{Proof of (2)(a):}}}  Fix $0 \ne v \in \GF(p^r)$ and consider the map $\GF(p^r) \rightarrow \GF(p^r)$ which maps $x \mapsto (x+ v)^{-1} - x^{-1}$.  The size of the image of this map is equal to the number of distinct $b$s that solve the equation $(x+ v)^{-1} - x^{-1} = b$.  If $x \ne \mathbf{0}$ or $-v$, then $(x+v)((x+ v)^{-1}) = 1$ and $x(x^{-1}) = 1$, and 
\begin{center}
\begin{tabular}{cl}
${}$& $(x+ v)^{-1} - x^{-1} = b$ \\
$\Leftrightarrow$& $x(x+v)((x+ v)^{-1} - x^{-1}) = x(x+v)b$ \\
$\Leftrightarrow$ & $x - (x+v) = bx^2 + bvx$ \\
$\Leftrightarrow$ & $bx^2 + bvx + v = 0$\\
$\Leftrightarrow$ & $b(x^2 + vx) = -v$\\
\end{tabular}
\end{center}
Now as $x$ ranges over $\GF(p^r)$ except $\mathbf{0}$ and $-v$, the quantity $(x^2 + vx)$ ranges over at least $\frac{p^r-2}{2}$ distinct nonzero values, whence solving for $b$ we find at least $\frac{p^r-2}{2}$ distinct values of $b$. Therefore the map $x \mapsto (x+ v)^{-1} - x^{-1}$ has at least $\frac{p^r-2}{2} > p^{r-2}$ distinct values, fulfilling condition 2(a).

{\flushleft{\underline{Proof of (2)(b):}}}   
Assume that $U$ is a proper (vector-) subspace of $V_i$ and $U$ is closed under inversion.   
As subspace, $U$ is an additive subgroup of $V_i$.  Apply Theorem \ref{mattareithm} and Lemma \ref{qfelemma} to find that either $U$ is a subfield of $V_i$, or $|U|=|F|$ for some subfield $F \subset V_i$.  Since $V_i$ is isomorphic to $\GF(p^r)$, Theorem \ref{fieldorderthm} implies that $|U| = p^k$ where $k\vert r \mbox{ and }k\ne {r}$. Then as $k$ is  a proper divisor of $r$, $k\le\frac{r}{2}$. But then we have the following implications
\begin{center}
\begin{tabular}{cll}
&$|U| \le p^{\frac{r}{2}}$ & \\
$\Rightarrow$ &$\mathrm{dim}(U) \le \frac{r}{2}$& because $|U|=p^{\mathrm{dim}(U)}$\\
$\Rightarrow$ &$\mathrm{codim}(U) \ge \frac{r}{2}$& because dim$(U) + $ codim$(U) = $ dim$(V_i) = r$\\
$\Rightarrow$ &$\mathrm{codim}(U) \ge 3$& provided $r \ge 5$.
\end{tabular}
\end{center}
This completes the proof of condition 2(b) and the theorem.
\end{proof}

\begin{theorem}\label{towardsprimitivity1} 
Let $r>4$ and  $V = M_{m,n}(\GF(p^r))$. If $U \neq \{ {\mathbf 0} \}$ is a subspace of $V$ such that for all $u\in U$ and  $v \in V$
\begin{equation*}
(\alpha \circ \gamma)(v+ u) - (\alpha \circ \gamma)(v) \in U,
\end{equation*} \
where $\alpha=\rho \circ \pi$,  then $U$ is invariant under $\alpha$ and $U$ is a sum of some of the $V_i$.
\end{theorem}

\begin{proof}
We already know that $\alpha$ is a permutation of the set $V$. By Lemma \ref {SRlinear} and Lemma \ref{MClinear} we have that $\alpha$ is an invertible linear transformation of the vector space $V$ over the field $\GF(p^r)$. Thus, $W = \alpha^{-1}\lbrack U\rbrack$ is a vector subspace of $V$ of the same dimension as $U$. 

Thus, for all $u \in U$ and $v \in V$ we have
\begin{equation}
\gamma(v + u) - \gamma(v) \in \alpha^{-1}\lbrack U\rbrack = W\text{.} \label{First}
\end{equation} \
Setting $v = 0$ in $\eqref{First}$ and using the fact that $\gamma(0)=0$, we see that for each $u\in U$ we have $\gamma(u) \in W$.  Hence, $\gamma$ is a function from $U$ to $W$. Since $U$ and $W$ are finite and $|\gamma\lbrack U\rbrack| = |U| = |W|$, (1) of Lemma \ref{1rdlemma} implies that
\begin{equation*}
\gamma\lbrack U\rbrack = W \text{ and } \gamma\lbrack W\rbrack = U\text{.}
\end{equation*} \
Using the hypothesis that $U$ is not $\{{\mathbf 0} \}$, choose a $u\in U$ and an $i$ such that $u_i = Proj_i(u) \neq 0$. With $i$ fixed from now on, consider any $v_i \in V_i$ with $v_i \neq 0$. We have that 
$\gamma(u + v_i) - \gamma(v_i) \in W$ and $\gamma(u) \in W$.  Since $W$ is a vector space, $-\gamma(u) + \gamma(u + v_i) - \gamma(v_i) \in W$.
Explicitly written $\gamma(u+v_i)$ and $\gamma(u)$ have the form
\[
  \gamma(u + v_i) = \gamma_1(u_1) \oplus \gamma_2(u_2) \oplus \cdots \oplus \gamma_i(u_i + v_i) \oplus \cdots \oplus \gamma_{mn}(u_{mn})
\]
and 
\[ \gamma(u) = \gamma_1(u_1) \oplus \gamma_2(u_2) \oplus \cdots \oplus \gamma_i(u_i) \oplus \cdots \oplus \gamma_{mn}(u_{mn}).
\]
Since $V_i$ is a vector space,  $-\gamma_i(u_i) + \gamma_i(u_i + v_i) - \gamma_i(v_i)\in V_i$. 
Therefore,
\[
  -\gamma(u) + \gamma(u + v_i) - \gamma(v_i) = -\gamma_i(u_i) + \gamma_i(u_i + v_i) - \gamma_i(v_i) \in W \cap V_i.
 \]  
If for each $v_i\in V_i$ this vector was the zero-vector, then the image of the map $v_i \mapsto \gamma_i(v_i + u_i) - \gamma_i(v_i)$ from $V_i$ to $V_i$ would be $\{ \gamma_i(u_i) \}$.  This would contradict (2)(a) of Lemma \ref{1rdlemma}.  Thus, $W \cap V_i \neq \{ {\mathbf 0} \}$.

Since  $U\cap V_i=\gamma(W \cap V_i)$  and $\gamma_i(x)=0$ implies $x=0$, we have that $U \cap V_i \neq \{ {\mathbf 0} \}$.  Thus there is  a non-zero element $u_i \in U \cap V_i$. By the hypothesis that $r>4$ and (2)(a) of Lemma \ref{1rdlemma}, the map $x \mapsto \gamma_i(x + u_i) - \gamma_i(x)$ from $V_i$ to $V_i$ has image of cardinality greater than $p^{r-2}$. But as seen in $\eqref{First}$, the image of this map is also a subset of $W$. Thus $W\cap V_i$ is a linear subspace of $V_i$ and has cardinality greater then $p^{r-2}$. As subspace of $V_i$ the cardinality of $W\cap V_i$ must be factor of the cardinality $p^r$ of $V_i$ and thus is a power of the prime number $p$. It follows that the cardinality of $W\cap V_i$ is at least $p^{r-1}$. But then the codimension of $W\cap V_i$ in $V_i$ is at most $1$.  Similarly, the codimension of $U \cap V_i$ is at most $1$.  Hence, the subspace $U \cap W \cap V_i$ of $V_i$ has codimension of at most $2$ in $V_i$. In particular, since $r>2$ we have that $U\cap W\cap V_i\neq \{ {\mathbf 0} \}$.  

Because $\gamma(U) = W$ and $\gamma(W) = U$, we see that $U \cap W \cap V_i$ is invariant under $\gamma$.  From Condition (2), it follows that $U \cap W \cap V_i = V_i$.  Hence, $U \supset V_i$.

So if $U$ contains an element of $V_i$ for some $i$, then $U \supset V_i$.  Hence, $U$ is a direct sum of some of the $V_i$.  Since $W = \gamma(U)$ and $\gamma(V_i) = V_i$ for all $i$, we see that $W = U$.  And since $U = \gamma(W)$, it follows that $U = \alpha(U)$
\end{proof}

\begin{theorem}\label{primthm}
Let $\tau= \{T[k]\vert k \in \mathcal{K}\}$ be the set of all generalized Rijndael-like functions  $T[k]: M_{m,n}(\GF(p^r))\rightarrow M_{m,n}(\GF(p^r))$  ($p\geq 2$) and  $\GG_{\tau}=\langle T[k]\vert k\in\KK\rangle$ be the group generated by the set $\tau$.   Assume  that the only subspaces of  $M_{m,n}(\GF(p^r))$ that are invariant under $\alpha=\rho \circ \pi$ are $\{{\mathbf 0}\}$ and $M_{m,n}(\GF(p^r))$.  Then for all $m$, $n$ and $r>4$ the group $\GG_\tau$  is primitive.
\end{theorem}

\begin{proof}
Let $V = M_{m,n}(\GF(p^r))$.  Suppose that $G_{\tau}$ acts imprimitively on $V$.  By  Corollary 4.1 of \cite{C}, there is a proper subspace $U$ of $V$ such that $U \neq \{ {\mathbf 0} \}$ and such that for all $u\in U$ and  $v \in V$
\begin{equation*}
(\alpha \circ \gamma)(v+ u) - (\alpha \circ \gamma)(v) \in U.
\end{equation*} \
By Theorem \ref{towardsprimitivity1}, $U$ is a direct sum of some of the $V_i$ and an invariant subspace of $\alpha$ (\emph{i.e.}, $U = \alpha(U)$).  But this contradicts the hypothesis that $\alpha$ has no non-trivial invariant subspaces.  Therefore, $G$ is primitive. 
\end{proof}
The following theorem follows directly from  Lemma \ref{Rodgerslemma} and Theorem \ref{primthm}. 
\begin{theorem}\label{mainthm}
Let $\tau= \{T[k]\vert k \in \mathcal{K}\}$ be the set of all generalized Rijndael-like functions  on $M_{m,n}(\GF(p^r))$ and $\GG_{\tau}=\langle T[k]\vert k\in\KK\rangle$ be the group generated by the set $\tau$.  If  $\{{\mathbf 0}\}$ and $M_{m,n}(\GF(p^r))$ are the only subspaces of  $M_{m,n}(\GF(p^r))$ that are invariant under $\alpha=\rho \circ \pi$ and $\GG_{\tau}$ contains an $m$-cycle with $2\leq m\leq (n-m)!$, then for all  $m$,  $n>1$ and $r>4$  the group $\GG_\tau$ is either the alternating group or the symmetric group acting on  $M_{m,n}(\GF(p^r))$.
\end{theorem}

Note that the hypothesis that $\alpha$'s only invariant subspaces are $\{\mathbf{0}\}$ and $M_{m,n}(\GF(p^r))$ implies that $\gcd(c_1,...,c_m,n) = 1$.  Indeed,  suppose that  $\gcd(c_1,...,c_m,n) = x>1$.  Consider an input 
${\mathbf a}\in M_{m,n}(p^r)$  for $\alpha$ with only one non-zero entry  
\[  
  {\mathbf a} = \left[\begin{tabular}{cccccc}
              1        & 0 & 0 & ... & 0 & 0 \\
              0        & 0 & 0 & ... & 0 & 0 \\
              $\vdots$ &   &   &     &   &   \\
              0        & 0 & 0 & ... & 0 & 0 \\
           \end{tabular}
      \right]
\]       
Note that under $\alpha$, the orbit of ${\mathbf a}$ will have its non-zero entries at column positions of form $1 + k\cdot x \le n$, $k\in\naturals$. Thus, no orbit element will have a nonzero entry in the second column. But then as  $\alpha$ is linear it has an invariant subspace consisting of members of $M_{m,n}(\GF(p^r))$ that have no nonzero entries in the second column. This is a subspace different from $\{\mathbf{0}\}$ and $M_{m,n}(\GF(p^r))$, contradicting that $\alpha$'s only invariant subspaces are  $\{\mathbf{0}\}$ and $M_{m,n}(\GF(p^r))$.

Also, note that in general the condition {$\gcd(c_1,...,c_m,n) = 1$ is not sufficient to guarantee that $\alpha$'s only invariant subspaces are $\{\mathbf{0}\}$ and $M_{m,n}(GF(p^r))$. To see this, the reader is invited to consider the following example.

{\bf Example.}  Consider the vector space $M_{2,8}(\GF(7))$, an irreducible polynomial $f(x) = x^2+x+3$ over $\GF(7)$ and  $c_1 = 1$ and $c_2 = 5$. Since the MixColumns-like function $\rho$ is linear it can be specified as ${\mathbf d}=M\cdot {\mathbf c}$ for $c, d \in M_{2,8}(\GF(7))$ and $M$  a matrix of dimension $2\times 2$.   Let 
\[
  M= \left[\begin{tabular}{cc}
               1 & 4 \\
               1 & 0 \\
        \end{tabular}
  \right]
\]
i.e. the generating polynomial $M(x)=x+1$ for $\GF(7)/\langle f\rangle$. 

Now let ${\mathbf a}\in M_{2,8}(\GF(7))$ 
\[
 \mathbf a = \left[\begin{tabular}{cccccccc}
            1 & 0 & 0 & 0 & 0 & 0 & 0 & 0 \\
            3 & 0 & 0 & 0 & 0 & 0 & 0 & 0 \\
      \end{tabular}
     \right]
\]
be the input in the function $\alpha$. 
 It is easy to see that  the orbit of $\mathbf a$ under alpha has $48$ elements containing a linearly independent subset of at most $15$ elements. Thus the subspace $W$ generated by this orbit has dimension $1< dim(W) \le 15$, and as $\alpha$ is linear, this is an invariant subspace of $\alpha$ with dimension less than $dim(M_{2,8}(\GF(7)))=16$.

Note that the ShiftRows-like function $\pi$ for this example (in the sense of Definition 9.4.1 of \cite{DRB}) and the MixColumns-like function $\rho$ (in the sense that the orbit of any non-zero column vector includes all the nonzero column vectors) are  diffusion optimal. Thus, merely requiring that ShiftRows is diffusion optimal is not sufficient to guarantee that the only invariant subspaces of $\alpha$ are $\{\mathbf{0}\}$ and $M_{m,n}(\GF(p^r))$. \\

Next we determine the group $\GG_{\tau}^s=\langle T[k_s]T[k_{s-1}]\cdots T[k_1]\vert k_i\in\KK\rangle$ generated  by the set of all compositions of $s$ (independently chosen) generalized Rijndael-like functions.
\begin{theorem}\label{groupsgenth}
Let $\tau= \{T[k]\vert k \in \mathcal{K}\}$ be the set of all generalized Rijndael-like functions and  $\GG_{\tau}=\langle T[k]\vert k\in\KK\rangle$ be the group generated by the set $\tau$.   Then 
\begin{itemize}
\item [(a)]  If $\GG_{\tau} = \mathcal{A}_{p^{rmn}}$, then $\GG_{\tau}^s = \mathcal{A}_{p^{rmn}}$.  
\item [(b)] If $\GG_{\tau} = \mathcal{S}_{p^{rmn}}$, then $\GG_{\tau}^s = \mathcal{A}_{p^{rmn}}$ if $s$ is even and $\GG_{\tau}^s = \mathcal{S}_{p^{rmn}}$ if $s$ is odd.
\end{itemize}
\end{theorem}

\begin{proof}
Part (a) follows immediately from Lemma \ref{NormalSubGrp} and Theorem \ref {simple}.  
To show Part (b) suppose that $\GG_{\tau} = \mathcal{S}_{p^{rmn}}$.  If $s$ is even, then every element of $\GG_{\tau}^s$ must be an even permutation.  Hence $\GG_{\tau}^s = \mathcal{A}_{p^{rmn}}$ by Lemma \ref{NormalSubGrp}.  If $s$ is odd, then $\GG_{\tau}^s$ must contain an odd permutation.   Hence $\GG_{\tau}^s = \mathcal{S}_{p^{rmn}}$, again by Lemma \ref{NormalSubGrp}. 
\end{proof}

\section {Conclusion}  

In this paper we  provided conditions for which the round functions of a Rijndael-like block cipher deployed over a finite field $\GF(p^r)$ ($p>2$) do not constitute a group under functional composition - Theorem \ref{fullpermutations}. We also provided conditions for which the round functions of a Rijndael-like block cipher  over a finite field $\GF(p^r)$ ($p\geq 2$) generate  either the alternating group or the symmetric group on the message space - Theorem \ref{groupsgenth}. 

\section {Acknowledgments}
{Authors would like to thank  to R\"{u}diger Sparr and  Ralph Wernsdorf  for their valuable comments.}

\vskip 10truept

\begin{flushleft}
 $^1$ \small {Department of Mathematics, Boise State University, Boise, ID 83725}\\
 $^2$ \small {Department of Mathematics, Statistics, and Physics, Wichita State University, Wichita, KS 67260}\\
 $^3$ \small {Department of Mathematics, University of Notre Dame, Notre Dame, IN 46556}\\
 $^4$ \small {Department of Mathematics, Washington University, St. Louis, MO 63130}\\
 $^5$ \small {Department of Mathematics and Computer Science, Colorado College, Colorado Springs, CO 80903}
\end{flushleft}

\end{document}